\documentclass[11pt]{amsart}
\usepackage{amsmath, amsthm, amscd, amsfonts, amssymb, graphicx, color}
\usepackage[colorlinks]{hyperref}
\newtheorem{thm}{Theorem}[section]
\newtheorem{cor}[thm]{Corollary}

\newtheorem{prop}[thm]{Proposition}
\theoremstyle{definition}
\newtheorem{defn}[thm]{Definition}
\theoremstyle{remark}

\numberwithin{equation}{section}

\newcommand{\h}{\mathcal{H}}
\newcommand{\F}{\mathcal{F}}
\newcommand{\K}{\mathcal{K}}

\begin{document}
\title[Frames and Operators in Hilbert $C^*$-modules]{Frames and Operators in Hilbert $C^*$-modules}%
\author{ Abbas Najati, M. Mohammadi Saem  and and P. G\u{a}vru\c ta }%
\address{ \newline
\indent Abbas Najati and M.M. Saem
\newline
\indent Department of Mathematics
\newline
\indent Faculty of Mathematical Sciences
\newline
\indent   University of Mohaghegh Ardabili
\newline \indent  Ardabil 56199-11367
\newline \indent Iran}
\email{a.najati@uma.ac.ir,~ a.nejati@yahoo.com~(\rm A. Najati)}
\email{m.mohammadisaem@yahoo.com~(\rm M. M. Saem)}
\address{
\newline
\indent  P. G\u avru\c ta
\newline
\indent Department of Mathematics
\newline \indent  Politehnica University of Timi\c soara
\newline \indent Pia\c ta Victoriei, Nr. 2, 300006, Timi\c soara
\newline \indent Romania}
 \email{pgavruta@yahoo.com}


\begin{abstract}
In this paper we introduce the concepts of atomic systems for operators and
$K$-frames in Hilbert $C^*$--modules and we establish some results.
\end{abstract}
\maketitle
\section{{\textbf INTRODUCTION}}
Frames for Hilbert spaces were introduced by Duffin and
Schaeffer\cite{ds} as part of their research in non-harmonic Fourier
series. A finite or countable sequence $\{f_n\}_{n\in I}$ is called
a frame for a separable Hilbert space $\h$ if there exist constants
$A, B
>0$ such that
\begin{equation}
A\|f\|^2\leqslant \sum_{n\in I} |\langle f,f_n\rangle|^2\leqslant
B\|f\|^2,\quad f\in \h.
\end{equation}
\par
The frames have many properties which make them very useful in applications. See \cite{ch}.

Frank and Larson \cite{fl1,fl2} extended this concept for countably
generated Hilbert $C^*$-modules.
\par
Let $A$ be a $C^*$-algebra and $\h$ be a  left $A$-module. We assume
that the linear operations of $A$ and $\h$ are comparable, i.e.
$\lambda(ax) = (\lambda a)x = a(\lambda x)$ for every $\lambda \in
\Bbb{C}, a \in A$ and $x\in\h.$ Recall that $\h$ is  a pre-Hilbert
$A$-module if there exists a sesquilinear mapping $\langle
.,.\rangle: \h\times \h\rightarrow A$ with the properties
\begin{enumerate}
\item $\langle x,x\rangle\geq 0$; if $\langle x,x\rangle=0$, then $x=0$ for every $x\in \h$.
\item $\langle x,y\rangle=\langle y,x\rangle^*$ for every $x,y\in
\h$.
\item $\langle ax,y\rangle=a\langle x,y\rangle$ for every $a\in
A$, $x,y\in \h$.
\item $\langle x+y,z\rangle=\langle x,z\rangle+\langle y,z\rangle$
for every $x,y,z\in \h$.
\end{enumerate}
The map $x\mapsto\|x\|=\|\langle x,x\rangle\|^\frac{1}{2}$ defines a
norm on $\h$. A pre-Hilbert $A$-module  is called a Hilbert
$A$-module if $\h$ is complete with respect to that norm. So $\h$
becomes the structure of a Banach $A$-module. A Hilbert $A$-module
$\h$ is called countably generated if there exists a countable set
$\{x_n\}_{n\in J}\subseteq \h$ such that the linear span (over
$\Bbb{C}$ and $A$) of this set is norm-dense in $\h$.
\par
Suppose that $\h,\K$ are Hilbert $A$-modules over a $C^*$-algebra
$A$. We define $L(\h,\K)$ to be the set of all maps $T:\h\rightarrow
\K$ for which there is a map $T^*:\K\rightarrow \h$ such that
\begin{eqnarray*}
\langle Tx,y\rangle=\langle x,T^*y\rangle\quad x\in \h, y\in \K.
\end{eqnarray*}
It is easy to see that each $T\in L(\h,\K)$ is $A$-linear and
bounded. $L(\h,\K)$ is called the set of adjointable maps from $\h$
to $\K$. We denote $L(\h,\h)$ by $L(\h)$. In fact $L(\h)$ is a
$C^*$-algebra.
\par

For basic results on Hilbert modules see \cite{BG, M, P}.

Throughout the present paper we suppose that $A$ is a unital
$C^*$-algebra and  $\h$ is a Hilbert $A$-module.
\begin{defn}
Let  $J\subseteq \Bbb{N}$ be a finite or countable index set. A
sequence $\{f_n\}_{n\in J}$ of elements of $\h$ is said to be a
\textit{frame} if there exist two constants $C, D> 0$ such that
\begin{equation}\label{fr}
C \langle x,x\rangle\leqslant \sum_{n\in J}\langle
x,f_n\rangle\langle f_n,x\rangle\leqslant D\langle x,x\rangle,\quad
x\in \h.
\end{equation}
The constants $C$ and $D$ are called the \textit{lower} and
\textit{upper} \textit{frame bounds}, respectively. We consider
\textit{standard frames} for which the sum in the middle of
(\ref{fr}) converges in norm for every $x\in \h$. A frame
$\{f_n\}_{n\in J}$ is said to be a \textit{tight frame}  if $C=D$,
and said to be a \textit{Parseval frame} (or a \textit{normalized
tight frame}) if $C=D=1$. If just the right-hand inequality in
(\ref{fr}) holds,  we say that $\{f_n\}_{n\in J}$ is a
\textit{Bessel sequence} with a \textit{Bessel bound} $D$.
\end{defn}
\par
It follows from the above definition that a sequence $\{f_n\}_{n \in
J}$ is a normalized tight frame if and only if
\[\langle x,x\rangle=\sum_{n\in J}\langle x,f_n\rangle \langle
f_n,x\rangle,\quad x\in \h.\]
\par
\hspace{-0.2cm}{Let $\{f_n\}_{n\in J}$ be a standard frame for $\h$.
The \textit{frame transform }for $\{f_n\}_{n\in J}$ is the map $T:\h
\rightarrow \ell^2(A)$ defined by $Tx=\{\langle x,f_n\rangle\}_{n\in
J}$, where $\ell^2(A)$ denotes a Hilbert $A$-module
{$\{\{a_j\}_{j\in J}:a_j\in A, \sum_j a_ja_j^*~ \mbox{\rm converges
in norm}\}$} with pointwise operations and the inner product
{$\langle \{a_j\}_{j\in J},\{b_j\}_{j\in J}\rangle =\sum_{j\in J}
a_jb_j^*$}}. The adjoint operator $T^*: \ell^2(A)\rightarrow  \h$ is
given by $T^*(\{c_j\}_{j\in J})=\sum_{j\in J}c_jf_j$ (\cite{fl2},
Theorem 4.4). By composing $T$ and $T^*$, we obtain the
\textit{frame operator} $S:\h\to\h$ given by
\[Sx=T^*Tx=\sum_{n\in J}\langle x,f_n\rangle f_n, \quad x\in
\h.\] The frame operator is positive and invertible, also it is the
unique operator in $L(\h )$ such that the reconstruction formula
\begin{eqnarray*}
x=\sum_{n\in J} \langle x,S^{-1}f_n\rangle f_n=\sum_{n\in
J}\langle x,f_n\rangle S^{-1}f_n,
\end{eqnarray*}
holds for all $x\in\h.$ It is easy to see that the sequence
$\{S^{-1}f_n\}_{n\in J}$ is a frame for $\h$. The frame
$\{S^{-1}f_n\}_{n\in J}$  is said to be the \textit{canonical dual
frame} of the frame $\{f_n\}_{n\in J}$.
\par
There exists Hilbert $C^*$-modules admitting no frames (see \cite{Li}). The Kasparov Stabilisation Theorem \cite{Kas} is used in \cite{fl2} to prove that every countably generated Hilbert Module over a unital $C^*$-algebra admits frames.
The following  Proposition gives an equivalent definition of frames
in Hilbert $C^*$--modules.
\begin{prop}\cite{W}
Let $\h$ be a finitely or countably generated Hilbert $A$-module and
$\{f_n\}_{n\in J}$ be a sequence in $\h$. Then $\{f_n\}_{n\in J}$ is
a frame of $\h$ with bounds $C$ and $D$ if and only if
\begin{eqnarray*}
C\|x\|^2\leqslant \Big\|\sum_{n\in J} \langle x,f_n\rangle\langle
f_n,x\rangle\Big\|\leqslant D\|x\|^2,
\end{eqnarray*}
for all $x\in \h$.
\end{prop}
\par
We recall that  an element $v\in\h$ is said to be a \textit{basic
element} if $e=\langle v,v\rangle$ is a minimal projection in $A$;
that is $eAe=\Bbb{C}e$. A system $\{v_i\}_{i\in J}$ of basic
elements of $\h$ is said to be \textit{orthonormal} if $\langle
v_i,v_j\rangle=0$, for all $i\neq j$; moreover if this orthonormal
system generates a dense submodule of $\h$, then we call it an
\textit{orthonormal basis} for $\h$.
\par
We need the following results to prove our results.
\begin{thm}\cite{fy}\label{t.1}
Let $\F,\h,\K$ be Hilbert $C^*$-modules over a $C^*$-algebra $A$.
Also let $S\in L(\K,\h)$ and $T\in L(\F,\h)$ with
$\overline{R(T^*)}$ orthogonally complemented. The following
statements are equivalent:
\begin{enumerate}
\item $SS^*\leqslant \lambda TT^*$ for some $\lambda>0$;
\item there exists $\mu>0$ such that $\|S^*z\|\leqslant
\mu\|T^*z\|$ for all $z\in \h$;
\item there exists $D\in L(\K,\F)$ such that $S=TD$, i.e.,
$TX=S$ has a solution;
\item $R(S){\subseteq}R(T)$.
\end{enumerate}
\end{thm}
\begin{prop}\cite{W}\label{p.1}
Let $\{f_n\}_{n\in J}$ be a sequence of a finitely or countably
generated Hilbert $C^*$-module $\h$ over a unital $C^*$-algebra $A$.
Then the following statements are mutually equivalent:
\begin{enumerate}
\item $\{f_n\}_{n\in J}$ is a Bessel sequence for $\h$ with bound
$D$.
\item $\Big\|\sum_{n\in J}\langle x,f_n\rangle\langle
f_n,x\rangle\Big\|\leqslant D\|x\|^2,\quad x\in \h$.
\item $\theta:\ell^2(A)\rightarrow \h$ defined by
\begin{eqnarray*}
\theta(\{c_n\}_{n\in J})=\sum_{n\in J} c_nf_n.
\end{eqnarray*}
is a well-defined bounded operator with $\|\theta\|\leqslant
\sqrt{D}$.
\item $T:\h \rightarrow \ell^2(A)$ defined by $Tx=\{\langle
x,f_n\rangle\}_{n\in J}$ is adjointable and $T^*=\theta$.
\end{enumerate}
\end{prop}
\begin{prop}\cite{P}\label{p.2}
Let $\h$ be a Hilbert $C^*$-module. If $T\in L(\h)$, then $\langle
Tx,Tx\rangle\leqslant \|T\|^2\langle x,x\rangle$ for every $x\in
\h$.
\end{prop}
\begin{prop}\cite{W}\label{p.n1}
Let $B$ be a $C^*$-algebra and $\{a_n\}_{n\in J}$ a sequence in $B$.
If $\sum_{n\in J}a_nb^*_n$ converges for all $\{b_n\}_{n\in
J}\in\ell^2(B)$, then $\{a_n\}_{n\in J}\in\ell^2(B)$.
\end{prop}

\newpage
In \cite{G}, L. G\u avru\c ta, presented a generalization of frames, named $K$-frames, which allows to reconstruct elements from the range of a linear and bounded operator in a Hilbert space. She also introduced the concept of atomic system for operators and gave new results and properties of $K$-frames in Hilbert spaces. See also \cite{X}.

In the present paper, we extend these results for frames in $C^*$-Hilbert modules.
\vspace{1cm}
\section{{\textbf ATOMIC SYSTEMS IN HILBERT $C^*$-MODULES}}

Let $J\subseteq\Bbb{N}$ be a finite or countable index set.
\begin{defn}\label{d.1}
A sequence  $\{f_n\}_{n\in J}$ of $\h$ is called an {\it atomic
system} for $K\in L(\h)$ if the following statements hold:
\begin{enumerate}
\item the series $\sum_{n\in J} c_nf_n$ converges for all $c=\{c_n\}_{n\in J}\in
\ell^2(A)$;
\item there exists $C>0$ such that for every $x\in \h$ there exists
$\{a_{n,x}\}_{n\in J}\in \ell^2(A)$ such that $\sum_{n\in J}
a_{n,x}a_{n,x}^*\leqslant C\langle x,x\rangle$ and $Kx=\sum_{n\in J}
a_{n,x}f_n$.
\end{enumerate}
\end{defn}
\begin{prop}\label{p.n12}
Let $\{f_n\}_{n\in J}$ be a sequence in $\h$ such that $\sum_{n\in
J} c_nf_n$ converges for all $c=\{c_n\}_{n\in J}\in \ell^2(A)$. Then
$\{f_n\}_{n\in J}$ is a Bessel sequence in $\h$.
\end{prop}
\begin{proof}
It is clear that $\sum_{n\in J} c_n\langle f_n,x\rangle$ converges
for all $c=\{c_n\}_{n\in J}\in \ell^2(A)$ and all $x\in\h$. Hence
$\{\langle x,f_n\rangle\}_{n\in J}\in \ell^2(A)$ by Proposition
\ref{p.n1}. Let us define $T:\ell^2(A)\to \h$ by $T(\{c_n\})_{n\in
J}=\sum_{n\in J} c_nf_n$. Therefore $T$ is bounded and the adjoint
operator is given by
\[T^*:\h\to\ell^2(A),\quad T^*(x)=\{\langle x,f_n\rangle\}_{n\in J}.\]
Since $T^*$ is bounded, we get that $\{f_n\}_{n\in J}$ is a Bessel
sequence in $\h$.
\end{proof}
\begin{prop}\label{p.n13}
Let $\{f_n\}_{n\in J}$ be a sequence in $\h$. Then $\{f_n\}_{n\in
J}$ is a Bessel sequence in $\h$ if and only if $\{\langle
x,f_n\rangle\}_{n\in J}\in \ell^2(A)$, for all $x\in\h$.
\end{prop}
\begin{proof}
It is clear that if $\{f_n\}_{n\in J}$ is a Bessel sequence in $\h$,
then $\{\langle x,f_n\rangle\}_{n\in J}\in \ell^2(A)$, for all
$x\in\h$. The converse follows from the Uniform Boundedness
Principle.
\end{proof}

In the following, we suppose that $\mathcal{H}$ is finite or countable generated Hilbert $C^*$-module.
\begin{thm}
 If $K\in L(\h)$, then there exists an atomic system for $K$.
\end{thm}
\begin{proof}
Let  $\{x_n\}_{n\in J}$ be a standard normalized tight frame for
$\h$. Since
\begin{eqnarray*}
x=\sum_{n\in J}\langle x,x_n\rangle x_n,\quad x\in \h,
\end{eqnarray*}
we have
\begin{eqnarray*}
Kx=\sum_{n\in J}\langle x,x_n\rangle Kx_n,\quad x\in \h.
\end{eqnarray*}
For $x\in\h$, putting $a_{n,x}=\langle x,x_n\rangle$ and $f_n=Kx_n$
for all $n\in J$, we get
\begin{align*}
\sum_{n\in J} \langle x,f_n\rangle\langle f_n,x\rangle&=\sum_{n\in
J} \langle
x,Kx_n\rangle\langle Kx_n,x\rangle\\
&=\sum_{n\in J}\langle K^*x,x_n\rangle \langle
x_n,K^*x\rangle=\langle
K^*x,K^*x\rangle\\
&\leqslant \|K^*\|^2\langle x,x\rangle.
\end{align*}
Therefore $\{f_n\}_{n\in J}$ is a Bessel sequence for $\h$ and we
conclude that the series $\sum_{n\in J} c_nf_n$ converges for all
$c=\{c_n\}_{n\in J}\in \ell^2(A)$ by Proposition \ref{p.1}. We also
have
\begin{eqnarray*}
\sum_{n\in J} a_{n,x}a_{n,x}^*=\sum_{n\in J} \langle
x,x_n\rangle\langle x_n,x\rangle=\langle x,x\rangle,
\end{eqnarray*}
which completes the proof.
\end{proof}
\begin{thm}\label{t.2}
Let $\{f_n\}_{n\in J}$ be a Bessel sequence for $\h$ and $K\in
L(\h)$. Suppose that  $T\in L(\h, \ell^2(A))$  is given by
$T(x)=\{\langle x,f_n\rangle \}_{n\in J}$ and $\overline{R(T)}$ is
orthogonally complemented. Then the following statements are
equivalent:
\begin{enumerate}
\item $\{f_n\}_{n\in J}$ is an atomic system for $K$;
\item There exist $C,B>0$ such that
\begin{eqnarray*}
C\|K^*x\|^2\leqslant \Big\|\sum_{n\in J} \langle x,f_n\rangle
\langle f_n,x\rangle\Big\|\leqslant B\|x\|^2;
\end{eqnarray*}
\item There exists $D\in L(\h, \ell^2(A))$ such
that $K=T^*D$.
\end{enumerate}
\end{thm}
\begin{proof}
$(1)\Rightarrow (2)$. For every $x\in \h$, we have
\begin{eqnarray*}
\|K^*x\|=\sup_{\|y\|=1}\|\langle
y,K^*x \rangle\|=\sup_{\|y\|=1}\|\langle Ky,x\rangle\|.
\end{eqnarray*}
\par
Since $\{f_n\}_{n\in J}$ is an atomic system for $K$, there exists
$M>0$ such that for every $y\in \h$ there exists
$a_y=\{a_{n,y}\}_{n\in J}\in \ell^2(A)$ for which $\sum_{n\in J}
a_{n,y}a_{n,y}^*\leqslant M\langle y,y\rangle$ and $Ky=\sum_{n\in J}
a_{n,y}f_n$. Therefore
\begin{align*}
\|K^*x\|^2&=\sup_{\|y\|=1}\|\langle
Ky,x\rangle\|^2=\sup_{\|y\|=1}\Big\|\langle \sum_{n\in J}
a_{n,y}f_n,x\rangle
\Big\|^2=\sup_{\|y\|=1}\Big\|\sum_{n\in J} a_{n,y}\langle f_n,x\rangle\Big\|^2\\
&\leqslant \sup_{\|y\|=1}\Big\|\sum_{n\in J} a_{n,y}
a_{n,y}^*\Big\|\Big\|\sum_{n\in J}\langle x,f_n\rangle\langle
f_n,x\rangle\Big\|\\
&\leqslant\sup_{\|y\|=1} M \|y\|^2\Big\|\sum_{n\in J} \langle x,f_n\rangle\langle f_n,x\rangle\Big\|\\
&= M\Big\|\sum_{n\in J} \langle x,f_n\rangle\langle
f_n,x\rangle\Big\|,
\end{align*}
for every $x\in \h$. So that
\begin{eqnarray*}
\frac{1}{M}\|K^*x\|^2\leqslant\Big\|\sum_{n\in J} \langle
x,f_n\rangle\langle f_n,x\rangle\Big\|, \quad x\in \h.
\end{eqnarray*}
Moreover, $\{f_n\}_{n\in J}$  is a Bessel sequence for $\h$. Hence
$(2)$ holds.
\par
$(2)\Rightarrow(3)$ Since $\{f_n\}_{n\in J}$ is a Bessel sequence,
we get $T^*e_n=f_n$, where $\{e_n\}_{n\in J}$ is the standard
orthonormal basis for $\ell^2(A)$. Therefore
\begin{align*}
C\|K^*x\|^2&\leqslant \Big\|\sum_{n\in J}\langle x,f_n\rangle\langle
f_n,x\rangle\Big\|=\Big\|\sum_{n\in J}\langle x,T^*e_n\rangle\langle
T^*e_n,x\rangle\Big\|\\
&=\Big\|\sum_{n\in J}\langle Tx,e_n\rangle\langle
e_n,Tx\rangle\Big\|=\|Tx\|^2, \quad x\in \h.
\end{align*}
\par
By  Theorem \ref{t.1}, there exists operator $D\in L(\h, \ell^2(A))$
such that $K=T^*D$.
\par
$(3)\Rightarrow(1)$ For every $x\in \h$, we have \[Dx=\sum_{n\in
J}\langle Dx,e_n\rangle e_n.\] Therefore \[T^*Dx=\sum_{n\in J}
\langle Dx,e_n\rangle T^*e_n,\quad x\in \h.\]Let $a_{n}=\langle
Dx,e_n\rangle$, so for all $x\in \h$ we get
\begin{eqnarray*}
\sum_{n\in J} a_{n}a_{n}^*=\sum_{n\in J}\langle Dx,e_n\rangle\langle
e_n,Dx\rangle=\langle Dx,Dx\rangle\leqslant \|D\|^2\langle
x,x\rangle.
\end{eqnarray*}
Since $\{f_n\}_{n\in J}$  is a Bessel sequence for $\h$, we obtain
that $\{f_n\}_{n\in J}$ is an atomic system for $K$.
\end{proof}
\begin{cor}
Let $\{f_n\}_{n\in J}$ be a frame for $\h$ with  bounds $C,D>0$ and
$K\in L(\h)$. Then $\{f_n\}_{n\in J}$ is an atomic system for $K$
with bounds $\frac{1}{C^{-1}\|K\|^2}$ and $D$.
\end{cor}
\begin{proof}
Let $S$ be the frame operator of $\{f_n\}_{n\in J}$. We prove that
the condition  $(2)$ of Theorem \ref{t.2} holds. Since
$\{S^{-1}f_n\}_{n\in J}$ is a frame for $\h$  with bounds
$D^{-1},C^{-1}>0$ and  $x=\sum_{n\in J}\langle x,f_n\rangle
S^{-1}f_n$ for all $x\in \h$, we get
\begin{align*}
\|K^*x\|^2&=\sup_{\|y\|=1}\|\langle
K^*x,y\rangle\|^2=\sup_{\|y\|=1}\Big\|\langle \sum_{n\in J}\langle
x,f_n\rangle K^*S^{-1}f_n, y\rangle\Big\|^2\\
&=\sup_{\|y\|=1}\Big\|\sum_{n\in J}\langle x,f_n\rangle\langle K^*S^{-1}f_n,y\rangle\Big\|^2\\
&\leqslant \sup_{\|y\|=1}\Big\|\sum_{n\in J}\langle
x,f_n\rangle\langle f_n,x\rangle\Big\|\Big\|\sum_{n\in J}\langle
Ky,S^{-1}f_n\rangle\langle S^{-1}f_n,Ky\rangle\Big\|\\
&\leqslant\sup_{\|y\|=1} C^{-1}\Big\|\sum_{n\in J}\langle
x,f_n\rangle\langle
f_n,x\rangle\Big\|\|Ky\|^2\\
&=C^{-1}\|K\|^2\Big\|\sum_{n\in J}\langle x,f_n\rangle\langle
f_n,x\rangle\Big\|.
\end{align*}
So
\begin{align*}
\frac{1}{C^{-1}\|K\|^2}\|K^*x\|^2\leqslant \Big\|\sum_{n\in J}
\langle x,f_n\rangle\langle f_n,x\rangle\Big\|\leqslant D\|x\|^2,
\quad x\in \h.
\end{align*}
Therefore $\{f_n\}_{n\in J}$ is an atomic system for $K$.
\end{proof}
\par
The converse of the above corollary holds when the operator $K$ is
onto.
\begin{cor}
Let $\{f_n\}_{n\in J}$ be an atomic system for $K$. If $K\in L(\h)$
is onto, then $\{f_n\}_{n\in J}$ is a frame for $\h$.
\end{cor}
\begin{proof}
By Proposition 2.1 from \cite{A},  $K\in L(\h)$ is surjective if and
only if  there is $M>0$ such that
\begin{eqnarray*}
M\|x\|\leqslant \|K^*x\|,\quad x\in \h.
\end{eqnarray*}
Since $\{f_n\}$ is an atomic system for $K$, by Theorem \ref{t.2},
there exsit $C,B>0$ such that
\begin{eqnarray*}
C\|K^*x\|^2\leqslant \Big\|\sum_{n\in J} \langle x,f_n\rangle
\langle f_n,x\rangle\Big\|\leqslant B\|x\|^2,\quad x\in \h.
\end{eqnarray*}
Therefore
\begin{eqnarray*}
M^2C\|x\|^2\leqslant \Big\|\sum_{n\in J} \langle x,f_n\rangle
\langle f_n,x\rangle\Big\|\leqslant B\|x\|^2,
\end{eqnarray*}
for all $x\in \h$.
\end{proof}
\section{{\textbf K-FRAMES IN HILBERT $C^*$-MODULES}}
\begin{defn}
Let  $J\subseteq\Bbb{N}$ be a finite or countable index set. A
sequence $\{f_n\}_{n\in J}$ of elements in a Hilbert $A$-module $\h$
is said to be a \textit{$K$-frame} $(K\in L(\h))$ if there exist
constants $C,D>0$ such that
\begin{eqnarray}\label{e.1}
C\langle K^*x,K^*x\rangle \leqslant \sum_{n\in J} \langle
x,f_n\rangle\langle f_n,x\rangle\leqslant D\langle x,x\rangle,\quad
x\in \h.
\end{eqnarray}
\end{defn}
\begin{thm}
Let $\{f_n\}_{n\in J}$ be a Bessel sequence for $\h$ and $K\in
L(\h)$. Suppose that  $T\in L(\h, \ell^2(A))$  is given by
$T(x)=\{\langle x,f_n\rangle \}_{n\in J}$ and $\overline{R(T)}$ is
orthogonally complemented. Then $\{f_n\}_{n\in J}$ is a $K$-frame
for $\h$ if and only if there exists a linear bounded operator
$L:\ell^2(A)\to \h$ such that $Le_n=f_n$ and $R(K)\subseteq R(L)$,
where $\{e_n\}_n$ is the orthonormal basis for $\ell^2(A)$.
\end{thm}
\begin{proof}
Suppose that (\ref{e.1}) holds. Then
 $C\|K^*x\|^2\leqslant \|Tx\|^2$ for all $x\in \h$. By Theorem
\ref{t.1}, there exists $\lambda>0$ such that
\begin{eqnarray*}
KK^*\leqslant \lambda T^*T.
\end{eqnarray*}
Setting $T^*=L$, we get $KK^*\leqslant \lambda LL^*$ and therefore
$R(K)\subseteq R(L)$.
\par
Conversely, since $R(K)\subseteq R(L)$, by Theorem \ref{t.1} there
exists $\lambda>0$ such that $KK^*\leqslant \lambda LL^*$. Therefore
\begin{align*}
\frac{1}{\lambda}\langle K^*x,K^*x\rangle&\leqslant \langle
L^*x,L^*x\rangle=\sum_{n\in J} \langle x,f_n\rangle\langle
f_n,x\rangle,\quad x\in\h.
\end{align*}
Hence $\{f_n\}_{n\in J}$ is a $K$-frame for $\h$.
\end{proof}
\par
In the following theorem we offer a condition for getting a frame
from a $K$-frame.
\begin{thm}
Let $\{f_n\}_{n\in J}$ be a $K$-frame for $\h$ with bounds $C,D>0$.
If the operator $K$ is surjective, then $\{f_n\}_{n\in J}$ is a
frame for $\h$.
\end{thm}
\begin{proof}
By Proposition 2.1 from \cite{A},  $K\in L(\h)$ is surjective if and
only if there is $M>0$ such that
\begin{eqnarray*}
M\langle x,x\rangle\leqslant \langle K^*x,K^*x\rangle,\quad x\in \h.
\end{eqnarray*}
Since $\{f_n\}_{n\in J}$ is a $K$-frame, we get from (\ref{e.1})
\begin{eqnarray*}
M C\langle x,x\rangle \leqslant C\langle K^*x,K^*x\rangle\leqslant
\sum_{n\in J} \langle x,f_n\rangle\langle f_n,x\rangle\leqslant
D\langle x,x\rangle,\quad x\in \h.
\end{eqnarray*}
\end{proof}
\begin{prop}
A Bessel sequence $\{f_n\}_{n\in J}$ of $\h$ is a $K$-frame with
bounds $C,D>0$ if and only if $S\geqslant C KK^*$, where $S$ is the
frame operator for $\{f_n\}_{n\in J}$.
\end{prop}
\begin{proof}
A sequence $\{f_n\}_{n\in J}$ is a $K$-frame for $\h$ if and only if
\begin{align*}
\langle CKK^*x,x\rangle=C\langle K^*x,K^*x\rangle\leqslant
\sum_{n\in J}\langle x,f_n\rangle\langle f_n,x\rangle=\langle
Sx,x\rangle\leqslant D\langle x,x\rangle,
\end{align*}
\end{proof}


\end{document}